\documentclass[aos,preprint]{imsart}

\usepackage[utf8]{inputenc}
\usepackage[T1]{fontenc}
\usepackage[english]{babel}
\usepackage{xcolor}
\usepackage{amsmath, latexsym, amssymb,amsthm,enumerate}
\usepackage{natbib}
\newtheorem{thm}{\bf Theorem}[section]
\newtheorem{prop}[thm]{\bf Proposition}
\newtheorem{cor}[thm]{\bf Corollary}
\newtheorem{defin}[thm]{\textsl{\bf Definition}{}}
\newtheorem{lem}[thm]{\bf Lemma}
\newtheorem{ex}{\bf Example}[section]

\newtheorem{rem}{\bf Remark}[section]
\newtheorem{rems}[rem]{\bf Remarks}

\startlocaldefs
\renewcommand{\le}{\leqslant}
\renewcommand{\ge}{\geqslant}  
\newcommand{\lec}{\preccurlyeq}
\newcommand{\gec}{\succcurlyeq}

 \newcommand{\norm}[1]{\left\lVert #1 \right\rVert}  
 \newcommand{\interior}[1]{\mathring{#1}}
\newcommand\eps{\varepsilon}

\newcommand\e{\mathrm{e}}
\newcommand\dd{\mathrm{d}}

\newcommand\lb{[\![}
\newcommand\rb{]\!]}

\newcommand\T{\mathcal{T}}
\newcommand\1{\mathbf{1}}
\newcommand\Glm{\mathcal{G}_{\le m}}

\newcommand\Gm[1]{\mathcal{G}_{m_{#1}}}
\newcommand\Gj{\mathcal{G}_{j}}
\newcommand\R{\mathbb{R}}
\newcommand\E{\mathbb{E}}

\newcommand\N{\mathbb{N}}

\newcommand\D[1]{\mathrm{Desc}(#1)}
\newcommand\C[1]{\mathrm{Child}(#1)}

\newcommand\W{\mathsf{W}}
\newcommand\sm{s_\mathrm{max}}
\newcommand\fs{\frak{s}}
\newcommand\fe{\frak{e}}
\newcommand\vs{\sigma}

\renewcommand\ln{\mathrm{Log}}

\newcommand\Pno{\mathbb{P}_{n,0}}
\newcommand\Eno{\mathbb{E}_{n,0}}

\newcommand\Pnu{\mathbb{P}_{n,u}}

\endlocaldefs

\begin{document}

\begin{frontmatter}

\title{Minimax rates for finite mixture estimation}
\runtitle{Minimax rates for finite mixture estimation}


\author{\fnms{Philippe} \snm{Heinrich}\ead[label=e1]{philippe.heinrich@math.univ-lille1.fr}}
\address{\printead{e1}}
\and
\author{\fnms{Jonas} \snm{Kahn}\corref{JK}\ead[label=e2]{jonas.kahn@math.univ-lille1.fr}}
\address{\printead{e2}}
\affiliation{Universit\'e Lille 1\\ Laboratoire Paul Painlev\'e B\^at. M2\\
  Cit\'e Scientifique \\59655 Villeneuve d'Ascq,  FRANCE}

\runauthor{P. Heinrich and J. Kahn}

\begin{abstract}
 \ We prove that under some regularity and strong identifiability
 conditions, around a mixing distribution with $m_0$ components, the optimal local minimax rate
 of  estimation  of a mixture with $m$ components is $n^{-1/(4(m-m_0) + 2)}$.  This corrects a
 previous paper  by   \citet{Chen} in The Annals of
 Statistics. 
\end{abstract}



\end{frontmatter}

\section{Introduction}
Let be $\left\{ f(x,\theta) \right\}_{\theta \in \Theta }$ be a family
of probability densities with respect to some $\sigma$-finite measure $\lambda$. The
parameter set $\Theta $ is always assumed to be a compact subset of
$\R$ with non-empty interior. A finite mixture model with $m$
components is given by 
\begin{equation}
  \label{fxG}
  f(x,G)=\int_{\Theta}f(x,\theta)\dd G(\theta)
\end{equation}
where $G$ is a $m$-points support distribution on $\Theta$, called the \emph{mixing distribution}. The class
of such $m$-mixing distributions $G$ is denoted by $\Gm{}$ and $\Glm$
will be the union of $\Gj$ for $j\in \lb 1,m\rb$. 

In Section~\ref{sec:lowerbound} we will show that a consistent estimator $\widehat{G}_n\in \Glm$
of an unknown mixing distribution $G_1$ can not converge uniformly faster than $n^{-1/(4(m-m_0) + 2)}$ in the neighborhood of $G_0\in\Gm{0}$,   in
the ($L^1$-)Wasserstein metric,  
 where $n$ is the sample
size. Recall that this metric can be defined by 
\begin{equation}
  \label{defW}
  W(G_1,G_2)=\int_{\R}|G_1(-\infty,t]-G_2(-\infty,t]|\dd t,
\end{equation}
and that by the Kantorovich-Rubinstein dual representation, 
\begin{equation}
  \label{dualW}
   W(G_1,G_2)=\sup_{|f|_{\mathrm{Lip}}\le 1}\int_{\Theta}f(\theta)\dd (G_1-G_2)(\theta).
\end{equation}
In Section~\ref{sec:upperbound}, we prove that the rate
$n^{-1/(4(m-m_0) + 2)}$ is optimal, under strong identifiability conditions.  Finally, Section~\ref{sec:class} exhibits natural families satisfying these strong identifiability conditions.

Some auxiliary or too long computations are
postponed to Appendix~\ref{app}.

\section{The optimal rate can not be better than $n^{-1/(4(m-m_0) + 2)}$}
\label{sec:lowerbound}

The main idea is to build families of mixing distributions $G_{n}(u)$ with the same $2(m - m_0) $ first moments, and $u n^{-1/2}$ as rescaled shifted $(2 (m - m_0) + 1)$-th moment. Hence the Wasserstein distance between $G_n(u_1)$ and $G_n(u_2)$ will be of order $n^{-1/(4(m-m_0) + 2)}$. They will need $n$ observations to be told apart. Theorem \ref{LAN} makes this precise. We first need a few tools.

We give a far-from-general definition of local asymptotic normality \citep{LeCam}, but it is sufficient for our purposes.
\begin{defin}
    \label{defLAN} Given densities $f_{n,u}$ with respect to a measure
    $\lambda$, consider the sequence of experiments $\mathcal{E} _n =\left\{ f_{n,u}, u\in \mathcal{U} _n \right\} $ with each point of $\mathbb{R} $ in $\mathcal{U}_n$ for $n$ large enough. Let $X$ have density $f_{n,0}$ and consider the log-likelihood ratios:
    \begin{align*}
        Z_{n, 0}(u) & = \ln \left( \frac{f_{n, u}(X)}{f_{n,0}(X)}  \right) .
\end{align*}
    Suppose that there is a positive constant $\Gamma $ and a sequence of random variables $Z_n$ with $Z_n \xrightarrow[]{d} \mathcal{N} (0, \Gamma )$, such that for all $u\in \mathbb{R} $:
    \begin{align}
    \label{ELAN}
    Z_{n, 0} (u) - u Z_n + \frac{u^2}{2} \Gamma & \xrightarrow[n\to\infty]{P}  0
\end{align}
The sequence of experiments is said \emph{locally asymptotically
  normal} (LAN) and \emph{converging} to the Gaussian shift experiment $\left\{ \mathcal{N} (u\Gamma, \Gamma ), u \in \mathbb{R}  \right\} $.
\end{defin}

Of course, here  $\xrightarrow[]{d}$ (resp. $\xrightarrow[]{P}$) stands for convergence in distribution (resp. in probability). Intuitively, (almost) anything that can be done in a Gaussian shift experiment can be done asymptotically in a locally asymptotically normal sequence of experiments. 

\begin{defin}
    \label{Eia}
Let $\left\{ f(x, \theta ) \right\}_{\theta\in\Theta} $ be a family of
densities with respect to a $\sigma$-finite measure $\lambda $.
    Let us consider, for $p\in \N$ and $q>0$, the functions:
    \begin{align}
        \label{Eiaeq}
        E_{p,q} : \qquad \Theta ^3 \quad &  \to [0,\infty] \nonumber\\
        \left(\theta _1, \theta _2, \theta _3\right) &  \mapsto \mathbb{E}_{\theta _1}\left|\frac{ f^{(p)}(x, \theta _2)}{f(x,\theta _3)} \right|^q .  
    \end{align}
    We say that the family of densities is $(p,q )$-smooth if $E_{p,q
    }$ is well-defined and continuous on $\Theta^3$, and if there exists
    $\varepsilon > 0$ such that for all $\theta_1$,
    \begin{align}
\label{proche}
        |\theta _2 - \theta _3| < \varepsilon & \implies E_{p,q }(\theta _1, \theta _2, \theta _3) < \infty.
    \end{align}
    \end{defin}

    \begin{ex}
        Let us consider an exponential family with natural parameter
        $\theta\in \Theta _0 $, so that $f(x, \theta ) = h(x) g(\theta
        ) \exp(\theta T(x))$, with $g \in C^{\infty}$. Consider
        $\Theta $ such that its $\varepsilon $-neighbourhood $\Theta
        \oplus B(0, \varepsilon )$ is included in $\Theta_0 $. Then
        $\left\{ f(x, \theta), \theta \in \Theta  \right\} $ is $(p,q
        )$-smooth for any $p$ and $q$. Indeed,
        \begin{eqnarray*}
          f^{(p)}(x, \theta _2) & = &
                                      h(x)\e^{\theta_2T(x)}\left[\sum_{k=0}^p\binom{p}{k}g^{(k)}(\theta_2)T^{p-k}(x)\right]\\
\frac{ f^{(p)}(x, \theta _2)}{ f(x, \theta _3)}& = &
                                      \frac{\e^{(\theta_2-\theta_3)T(x)}}{g(\theta_3)}\left[\sum_{k=0}^p\binom{p}{k}g^{(k)}(\theta_2)T^{p-k}(x)\right]\\
\left|\frac{ f^{(p)}(x, \theta _2)}{ f(x, \theta _3)}\right|^q& = &
                                       a\frac{\e^{q(\theta_2-\theta_3)T(x)}}{g^q(\theta_3)}\left|\sum_{k=0}^p\binom{p}{k}g^{(k)}(\theta_2)T^{p-k}(x)\right|^q
\end{eqnarray*}
so that 
\[E_{p,q}(\theta _1, \theta _2, \theta _3)  = \frac{g(\theta_1) \E_{\theta_1+q(\theta_2-\theta_3)}\left|\sum_{k=0}^p\binom{p}{k}g^{(k)}(\theta_2)T^{p-k}(x)\right|^q}{g^q(\theta_3)g(\theta_1+q(\theta_2-\theta_3))}.\]
        Since all the moments of the sufficient statistic $T(x)$ are finite under a distribution in the exponential family, and since $\theta _1 + q\theta _2 - q \theta _3$ is in $\Theta _0$ for $ (\theta _2 -  \theta _3) < \eps/q$, we have finiteness of $E_{p,q } (\theta _1, \theta _2, \theta _3) $. Continuity is clear.
    \end{ex}

   
 Being $(p, q )$-smooth ensures finiteness of similar integrals when some $\theta _j$ are replaced with mixing distributions with components close to the $\theta _j$:
    \begin{prop}
        \label{tversmix}
        Given $\pi_0>0$ and two positive integers $m_0\le m$, define
        mixing distributions
 \[G_n = \sum_{j=1}^m \pi_{j,n} \delta _{\theta _{j,n}}\] 
such that $\theta _{j,n}\to \theta _0$ for all $ j \in\lb m_0,m\rb$ and $\sum_{j=
          m_0}^m \pi_{j,n} \geq \pi_0$ for all $n$ large enough. 
Consider a $(p, q)$-smooth family of densities
$\{f(x,\theta)\}_{\theta\in\Theta} $ with respect to some $\sigma$-finite measure $\lambda$. 

Then there is a finite $C$ depending only on  $\theta _0$ and $\pi_0$ such that for any $\theta$ satisfying  $\left|\theta - \theta _0 \right| \le \eps / 2$, for $n$ large enough, for any mixture $f(x,G)$:
        \begin{align*}
            \mathbb{E}_{G}\left|\frac{f^{(p)}(x, \theta)}{f(x,G_n)} \right|^q & \le C.
        \end{align*}
If, in addition, the function $\left|f^{(p)}(x, \theta _0)\right|$ has
nonzero integral under $\lambda$, then there is a positive $c$ depending only on  $\theta _0$ such that   for any mixture $f(x,G)$:
\begin{align*}
    \mathbb{E}_{G}\left|\frac{f^{(p)}(x, \theta _0)}{f(x,G)} \right|^q & \geq c.
\end{align*}
    \end{prop}
    \begin{proof}
      For $n$ large enough, we have $\left|\theta _{j,n} - \theta _0
      \right|\le \eps / 2$ for all $ j\in\lb m_0,m\rb$. Hence
      $\left|\theta _{j,n} - \theta\right|\le \eps$ for all $\theta$
      such that $\left|\theta - \theta _0 \right| \le \eps / 2$. So
      that we may use \eqref{proche}. By compactness and continuity,
      there is a finite $C$ such that  
\[\mathbb{E}_{\theta_1}\left|\frac{f^{(p)}(x,
    \theta)}{f(x, \theta _{j,n})} \right|^q \le C\]
 for all such $(j,n)$ and all $\theta _1$. Since $f(x,G)$ is a convex
 combination of some $f(x,\theta_1)$, we may replace $\theta _1$ by
 $G$ in the former expression. Since the function $1/y^q$ is
 convex on positive reals, by Jensen inequality, setting
 $A=\sum_{j=m_0}^m \pi_{j,n}$, 
 \[ \sum_{j=m_0}^m \frac{\pi_{j,n}}{A} \left| \frac{f^{(p)}(x, \theta)}{f(x,
     \theta _{j,n})} \right|^q\ge\left|
   \frac{f^{(p)}(x, \theta)}{\sum_{j=m_0}^m \frac{\pi_{j,n}}{A}f(x, \theta _{j,n})}\right|^q \geq A^q\left|
   \frac{f^{(p)}(x, \theta)}{f(x, G_{n})}\right|^q,\]
and taking expectations with respect to $G$ we obtain the upper bound 
\[\E_G\left|
   \frac{f^{(p)}(x, \theta)}{f(x, G_{n})}\right|^q\le \frac{C}{A^q}\le \frac{C}{\pi_0^q}.\]

The lower bound does not depend on $(p, q )$-smoothness. It is a simple consequence of rewriting:
        \begin{align*}
            \mathbb{E}_{G}\left|\frac{ f^{(p)}(x, \theta _0)}{f(x,G)} \right|^q & =  \int\left|\frac{ f^{(p)}(x, \theta _0)^q}{f(x,G)^{q-1}} \right|   \dd\lambda(x)     
        \end{align*}
        and noticing $\int\left|f(x,G) \right|\dd\lambda(x) =1$ since
        $f(x, G)$ is a probability density. By assumption, there is a
        set $B$ of measure $\lambda(B)=M>0$ on which the function $f^{(p)}(x, \theta
        _0)$ is more than some $\eps >0$. Now, the set $B\cap
        \{f(x,G)\le 2/M\}$ is of measure at least $M/2$ and thus 
\[\int\left|\frac{ f^{(p)}(x, \theta _0)^q}{f(x,G)^{q-1}} \right|
\dd\lambda(x)     \ge \left[\frac{M}{2}\right]^{q+1} \eps^q.\]
    \end{proof}

    \begin{thm}
        \label{LAN}
       Let $m_0 \leq m$. Let $G_0  = \sum_{j=1}^{m_0}  \pi_j \delta_{\theta _j} \in \mathcal{G} _{m_0}$ be a mixing distribution whose $m_0$-th component is
   in the interior of $\Theta $, that is 
$\theta _{m_0} \in \interior{\Theta
    }$.
  
    Then there are mixing distributions $G_n(u)$ ($n\ge 0,u\in\R$)
    all in
    $\Gm{}$ such that: 
    \begin{enumerate}[(i)]
        \item \label{lan1} $\W(G_{n}(u), G_0) \to 0$ for all $u\in\R$. More precisely,
      for some $C(u)>0$, we have 
\[\W(G_{n}(u), G_0) \le C(u) n^{-1/ (4(m -m_0) +2)}.\]
    \item \label{lan2} The mixing distributions get closer at rate $n^{-1/(4(m
        - m_0) + 2)} $: for all $u_1$ and $u_2$, there are constants
      $c(u_1,u_2)>0$ such that
      \[\W(G_{n}(u_1), G_{n}(u_2)) \geq c(u_1,u_2) n^{-1/ (4(m -m_0) +
        2)}.\]
    \item \label{lan3} Suppose that a family of densities $\left\{ f(x, \theta ), \theta \in
          \Theta  \right\} $ with respect to $\lambda$ is $(p,q )$-smooth for
        all $p \in\lb 1, 2(m-m_0+1)\rb$  and $q\in \lb 1, 4\rb$. Assume moreover that 
\[\int \left|f^{(2(m-m_0)+1)}(x,\theta_{m_0})\right|\dd\lambda(x)>0.\]
 There is a number $\Gamma >0$ and an infinite subset $\N_0$ of $\mathbb{N} $ along
      which the experiments $\mathcal{E} _n = \left\{
        \prod_{i=1}^nf\left(x_i,G_n(u)\right), |u|
        \le u_{\max}(n)\right\} $ with $u_{\max}(n) \to\infty$ converge to the Gaussian shift
      experiment $\left\{ \mathcal{N} (u\Gamma, \Gamma), u \in \mathbb{R}
      \right\}$.
  \item \label{lan4} $u$ is the rescaled $(2 (m - m_0) + 1)$-th moment of
    the components of the mixing distribution near $\theta
    _{m_0}$.
    \end{enumerate}
    \end{thm}

    The theorem shows that when the first moments of the components of the mixing distribution $G$ near $\theta _{m_0}$ are known, all remaining knowledge we may acquire is on the next moment, and that's the ``right'' parameter: it is exactly as hard to make a difference between, say, $10$ and $11$ as between $0$ and $1$.

On the other hand, for our original problem the cost function is the transportation distance between mixing distributions. So that an optimal estimator in mean square error for $u$ is not optimal for our original problem. Moreover just taking the loss function $  c(u_1,u_2)$ in the limit experiment runs into technical problems since this might go to zero as $u_2$ goes to infinity. They could be overcome, but it is easier to state a lower bound on risk using just contiguity and two points:
\begin{cor}
    \label{lower_bound}
    The optimal local minimax rate of estimation around $G_0$ of a mixture cannot be better than $ n^{-1/ (4(m - m_0) + 2)} $ in general: for any sequence of estimators $\hat{G}_n$ and any $\epsilon>0 $, we have:
    \begin{align}
        \label{local_minimax}
        \liminf_{n\to \infty} \!\!\!\!\!\!\sup_{\substack{G_1
      \text{s.t.}\\ W(G_1, G_0) < n^{-1/ (4(m - m_0) + 2)+\eps}}}
      \!\!\!\!\!\! n^{1/ (4(m - m_0) + 2)} \mathbb{E}_{f(\cdot,G_1)^{\otimes n}}\W(G_1, \hat{G}_n) & > 0,
    \end{align}
where the true distribution $G_1$ lies in $\Gm{}$.
\end{cor}
\begin{proof}[Proof of corollary~\ref{lower_bound}]

    Fix $u > 0$ and consider the densities $f_{n,u}(x)=\prod_{i=1}^nf\left(x_i,G_n(u)\right)$  with associated probability
    measures $\Pnu$ as in Theorem~\ref{LAN} \eqref{lan3}. We have 
    \begin{equation}
      \label{eq:liminf}
      \liminf_{n\to \infty}\inf_{A:\Pno(A)\ge 3/4}\Pnu(A)\ge
\frac14\e^{-\frac{u^2}{2}\Gamma}.
    \end{equation}
Indeed, the LAN property \eqref{ELAN} can be written as 
\[\rho_n:=\frac{f_{n,u}(X)}{f_{n,0}(X)}\e^{-uZ_n+\frac{u^2}{2}\Gamma}\xrightarrow[]{P}1,\]
with $X$ of density $f_{n,0}$ and $Z_n$ with asymptotic distribution $\mathcal{N}(0,\Gamma)$. For any event $A$, 
\begin{equation*}
  \Pnu(A)  = \Eno\left(\frac{f_{n,u}(X)}{f_{n,0}(X)}\1_A\right) = \Eno\left(\rho_n\e^{uZ_n-\frac{u^2}{2}\Gamma}\1_A\right).
\end{equation*}
Furthermore, by restriction on the event $\{Z_n>0\}$ and by using
$\rho_n \xrightarrow[]{P}1$, we get that $\Pnu(A)$ is bounded below by
\begin{equation*}
 \e^{-\frac{u^2}{2}\Gamma}\left[\Pno(A)-\Pno(Z_n\le 0)\right]+o(n).
\end{equation*}
Taking now the infimum on events $A$ such that $\Pno(A)\ge 3/4$ and passing to the limit as $n\to \infty$, we obtain \eqref{eq:liminf}.

We now consider, for any sequence of estimators $\hat{G}_n$, the event 
\[A =  \{ n^{1/ (4(m - m_0) + 2)} \W(G_{n}(0), \hat{G}_n) \geq a\}\] 
for some $a>0$ to choose.  Notice that by the triangle's inequality its complement $A^c$ satisfies 
\[A^c\subset  \{n^{1/ (4(m - m_0) + 2)} \W(G_{n}(u), \hat{G}_n) \geq
c(u,0) -a\}\]
where $c(u,0)>0$ is given by Theorem~\ref{LAN} \eqref{lan2}. Choose
$a=c(u,0)/2$.Then either  $\Pno(A) \geq 1/4  $, which gives 
\[\sup_{G_1 \in \{G_n(0)\}} n^{1/ (4(m
  - m_0) + 2)}\mathbb{E}_{f(\cdot,G_1)^{\otimes n}}\W(G_1, \hat{G}_n)
\ge \frac{a}{4},\]
 or  $\Pnu(A^c) \geq   \e^{-\frac{u^2}{2}\Gamma}/4$ in
the limit, by \eqref{eq:liminf},  so that 
\[ \liminf_{n\to \infty} \sup_{G_1 \in \{G_n(u)\}} n^{1/ (4(m
  - m_0) + 2)}\mathbb{E}_{f(\cdot,G_1)^{\otimes n}}\W(G_1, \hat{G}_n)
\ge \frac{a}{4}\e^{-\frac{u^2}{2}\Gamma}.\]
Thus,  gathering the two inequalities, we get
\[ \liminf_{n\to \infty} \sup_{G_1 \in \{G_n(0), G_n(u)\}} n^{1/ (4(m
  - m_0) + 2)}\mathbb{E}_{f(\cdot,G_1)^{\otimes n}}\W(G_1, \hat{G}_n)
\ge \frac{a}{4}\e^{-\frac{u^2}{2}\Gamma}.\]
Note to finish that by Theorem~\ref{LAN} \eqref{lan1}, each $G_n(0)$ or
$G_n(u)$ is at $\W$-distance at most  $n^{-1/ (4(m - m_0) + 2)+\eps}$
from $G_0$, for large $n$ enough. 
\end{proof}

\begin{rems}
    \label{remLAN}
    We want only an example of this slow convergence, and that it be somewhat typical. That's why we have chosen the regularity conditions to make the proof easy, while still being easy to check, in particular for exponential families. 
    
    In particular, it could probably be possible to lower $q$ in
    $(p,q)$-smoothness to $2+\eps$ and still get the uniform bound we
    use in the law of large numbers below. Similarly, less derivability might be necessary if we tried to imitate differentiability in quadratic mean.
    
    In the opposite direction the variance $\Gamma $ in the limit
    experiment is really expected to be $\pi_{m_0}^2\mathbb{E}_{G_0}
    \left| \frac{f^{(2d-1)}(x, \theta _{m_0}) }{f(x, G_0)} \right|^2$ in most cases, but more stringent regularity conditions may be needed to prove it.
\end{rems}
 \begin{proof}[Proof of  Theorem~\ref{LAN}]

In this proof and the rest of the paper, we need to compare asymptotic sequences.
The notation $a_n\lec b_n$ (or even $a\lec b$ if $n$ is kept
implicit) means that there is a  positive constant
$C$ such that $a_n\le C b_n$ ; in other words,
$a_n=O(b_n)$. We will also use $a_n \gec b_n$ for $a_n \ge C b_n$, and  $a_n\asymp b_n$  for $b_n \lec a_n\lec b_n$. Finally $a_n \lec_u b_n$ means that the constant may depend on $u$, that is  $a_n\le C(u) b_n$.

 We use the following theorem by  \citet[Theorem 2A]{Lindsay} on
 the matrix of moments ; the idea is close to the Hankel criterion
   developed by    \cite{Gass} to estimate the order of a mixture.
\begin{thm}
    \label{Lindsay}
Given numbers $1,m_1,\ldots,m_{2d}$, write $M_k$ for the
 $k+1$ by $k+1$ (Hankel)  matrix with entries $(M_k)_{i,j} =
 m_{i+j-2}$ for $k=1,\ldots,d$. 
    \begin{enumerate}[(a)]
    \item The numbers $1, m_1, \dots, m_{2d}$ are the moments of a
    distribution with exactly $p$ points of support if and only if
    $\det M_k > 0$ for $k=1,\ldots,d-1$ and $\det M_p = 0$.
  \item If the numbers $1, m_1, \dots, m_{2d-2}$ satisfies $\det M_k > 0$ for $k=1,\ldots,d-1$ and $m_{2d-1}$ is any scalar, then there exists a unique distribution with exactly $d$ points of support and those initial $2d-1$ moments. 
    \end{enumerate} 
\end{thm}
 Set $d=m-m_0+1$ and consider any  numbers $1, m_1, \dots,
 m_{2d-2}$ such that $\det M_1 > 0, \dots, \det M_{d-1} > 0$. By Theorem~\ref{Lindsay}, we may then define for any $u\in \mathbb{R} $ a distribution $G(u) = \sum_{j= m_0}^m \pi_j(u) \delta _{h_j(u)}$ such that its initial moments are $1, m_1, \dots, m_{2d-2}, u$.

Moreover, the unicity in Theorem \ref{Lindsay} implies that, with $\pi_i > 0$ and $h_1 < \dots < h_{d}$, the following application
is injective:
\begin{align*}
\phi :(\pi_1,\ldots,\pi_d,h_1,\ldots,h_d)\mapsto 
\left(\sum_1^d\pi_j,\sum_1^d\pi_j h_j,\sum_1^d\pi_j h_j^2,\ldots,\sum_1^d\pi_j h_j^{2d-1}\right)
\end{align*}
Now, its Jacobian is non-zero (see Appendix~\ref{jaco} for a proof): 
\begin{align}
    \label{Jac}
J(\phi)=(-1)^{\frac{(d-1)d}{2}}\,\pi_1\cdots\pi_d \prod_{1\le j<k\le d}(h_j- h_k)^4.  
\end{align}
Thus the inverse of $\phi$ is locally continuous, so that the $h_j(u)$
are all continuous. In particular, they are bounded if $u$ is bounded:
for any $U>0$, there is a finite $H(U)$ such that if $|u| < U$, then
$|h_j(u)| \le H(U)$. We may then find and use a sequence $u_{\max}(n)$ such that $u_{\max}(n) \to \infty$ and $H(u_{\max}(n)) n^{-1/(4d-2)} \to 0$. 

We now define the mixing distributions
\begin{equation}
  \label{mixdist}
 G_n(u)  = \sum_{j=1}^{m_0 - 1} \pi_j \delta_{\theta _j} + \pi_{m_0}
\sum_{j= m_0}^m \pi_j(u) \delta _{\theta_{j,n}(u)}
\end{equation}
with 
\[\theta _{j,n}(u)  = \theta _{m_0} + n^{-1/(4d-2)} h_j(u).\]
This definition satisfies \eqref{lan4}. The form of $G_n(u)$ makes it clear that it converges to $G_0$ at
speed  $n^{-1/(4d - 2)} $: it is easily seen from the dual
representation of $\W$ that for  $|u|\le U$
\[\W(G_{n}(u), G_0)\le \pi_{m_0}H(U) n^{-1/(4d - 2)}.\]
This proves \eqref{lan1}.

Moreover, since all other points and proportions are equal, the
transportation distance  $ \W(G_{n}(u_1), G_{n}(u_2))$ is equal to the
transportation distance between the last $p$ components. Since those
support points keep the same weights and are homothetic with scale
$n^{-1/(4p - 2)}$ around $\theta _{m_0}$, we have exactly 
\[ \W(G_{n}(u_1), G_{n}(u_2)) =  \W(G_{1}(u_1), G_{1}(u_2)) n^{-1/(4d - 2)}.\]
This proves \eqref{lan2}.

We now prove local asymptotic normality. In order to shorten
notations, the probability under the mixing distribution $G_n(0)$ will
be denoted by $\Pno$ and the corresponding expectation
$\Eno$. Let $X_{1,n}, \dots, X_{n,n}$ be an
i.i.d. sample with density $\prod_{i=1}^nf\left(x_i,G_n(0)\right)$. Then, we can write the Log-likelihood ratio as
\begin{equation*}
    Z_{n,0}(u)=\ln \left(\frac{\prod_{i=1}^nf(X_{i,n}, G_n(u))}{\prod_{i=1}^nf(X_{i,n}, G_n(0))}\right)=\sum_{i=1}^n \ln \left( 1 + Y_{i,n} \right).
\end{equation*}
with
\begin{equation}
\label{Yinu}
  Y_{i,n} (u)= \frac{f(X_{i,n}, G_n(u)) -  f(X_{i,n}, G_n(0)) }{ f(X_{i,n}, G_n(0))}.
\end{equation}
By definition, we have
\begin{equation*}
 f(x,G_n(u))- f(x,G_0)=\pi_{m_0}\sum_{j=m_0}^m\pi_{j,n}(u)\left[f(x,\theta_{j,n}(u))-f(x,\theta_{m_0})\right].
\end{equation*}
Moreover, by Taylor expansion with remainder, 
\begin{align*}
  f(x,\theta_{j,n}(u))-f(x,\theta_{m_0})=\sum_{k=1}^{2d-1}&\left(\frac{h_j(u)}{n^{1/(4d-2)}}\right)^kf^{(k)}(x,\theta_{m_0})\\
&+\int_{\theta_{m_0}}^{\theta_{j,n}(u)}f^{(2d)}(x,\theta)\frac{(\theta_{j,n}(u)-\theta)^{2d-1}}{(2d-1)!}\dd\theta
\end{align*}
so that  we get by linearity
\begin{align}  \label{FxGnu}f(x,G_n(u))-f(x,G_0)=\pi_{m_0}\left[\sum_{k=1}^{2d-1}\frac{m_k}{n^{k/(4d-2)}}f^{(k)}(x,\theta_{m_0})+R_n(x,u)\right]
\end{align}
with moments $m_1,\ldots,m_{2d-2}$ that do not depend on $u$ but
$m_{2d-1}=u$ and 
\begin{equation}
  \label{Rnxu}
 R_n(x,u)=\sum_{j=m_0}^m\pi_{j,n}(u) \int_{\theta_{m_0}}^{\theta_{j,n}(u)}f^{(2d)}(x,\theta)\frac{(\theta_{j,n}(u)-\theta)^{2d-1}}{(2d-1)!}\dd\theta.
\end{equation}
Thus, we can write from \eqref{Yinu}, \eqref{FxGnu} and \eqref{Rnxu}
\begin{equation}
\label{Yinubis}
 Y_{i,n}(u)= \pi_{m_0}\left[u n^{-1/2}Z_{i,n}+R_{i,n}(u)-R_{i,n}(0)\right] 
\end{equation}
with
\begin{equation*}
 R_{i,n}(u)=\frac{R_n(X_{i,n},u)}{f(X_{i,n}, G_n(0))}, \quad Z_{i,n}=\frac{f^{(2d-1)}(X_{i,n},\theta_{m_0})}{f(X_{i,n}, G_n(0))}.
\end{equation*}

For each fixed $n$ and $u$, the  $(Y_{i,n}(u),Z_{i,n},R_{i,n}(u)) $ are i.i.d. and centered under $G_n(0)$. Indeed, from \eqref{Yinu}, we have 
\[\Eno Y_{i,n}(u)=\int
[f(x,G_n(u))-f(x,G_n(0))]\dd\lambda(x)=0;\]
furthermore by expanding $f$ around $\theta_{m_0}$, we get iteratively using
$(p,q)$- smoothness 
that for $k=1,\ldots,2d-1$
\[\Eno \left[\frac{f^{(k)}(X_{i,n},\theta_{m_0})}{f(X_{i,n},G_n(0))}\right]=0\]
and in particular $\Eno Z_{i,n}=0$. And dividing \eqref{FxGnu} by
$f(x,G_n(0))$ gives as a result $\Eno R_{i,n}(u)=0$ for all $u$. 

Consider
\begin{equation}
\label{Zn}
  Z_n=\pi_{m_0}n^{-1/2}\sum_{i=1}^nZ_{i,n}.
\end{equation}
By Proposition~\ref{tversmix}, there are positive finite constants $c$ and $C$ independent on $n$ for $n$ large enough such that $c \le \Eno \left|Z_{1,n}\right|^2 \le C $. Up to taking a subsequence, we may then assume  $\Eno \left|Z_{1,n}\right|^2\to \sigma ^2$ for some positive $\sigma $. By Proposition \ref{tversmix} again, we have $ \Eno \left|Z_{1,n}\right|^3\le C'< \infty$ for all $n$ large enough.

We may then apply Lyapunov theorem \cite[Theorem 23.7]{Bill} to prove that, with $\Gamma = \sigma ^2\pi_{m_0}^2$,
\begin{align}
    \label{cvZn}
    Z_n &\xrightarrow[]{d} \mathcal{N} (0, \Gamma ).    
\end{align}
Indeed, setting $ s_n^2:=\sum_{i=1}^n\Eno
\left|Z_{i,n}\right|^2\sim n\sigma^2$, we see that the Lyapunov condition 
\[s_n^{-3}\sum_{i=1}^n \Eno \left|Z_{i,n}\right|^3\sim n^{-1/2}\sigma^{-3}\Eno
\left|Z_{1,n}\right|^3\xrightarrow[n\to\infty]{} 0\]
is satisfied so that $s_n^{-1}\sum_{i=1}^nZ_{i,n}$ converges in
distribution to $\mathcal{N} (0,1)$ and \eqref{cvZn} follows from the
equality $Z_n=\pi_{m_0}\left[\Eno
  \left|Z_{1,n}\right|^2\right]^{1/2}s_n^{-1}\sum_{i=1}^nZ_{i,n}$.

Now, to get the convergence  in probability of
$Z_{n,0}-uZ_n+\frac{u^2}{2}\Gamma$ to zero, it's enough to show the following convergences for all $u$:
\begin{eqnarray}
\sum_{i=1}^nY_{i,n}(u) -uZ_n & \xrightarrow[]{L^2} & 0, \label{p1}\\
  \sum_{i=1}^nY_{i,n}(u)^2-u^2\Gamma& \xrightarrow[]{L^1}  &0, \label{p2}\\
 \sum_{i=1}^n|Y_{i,n}(u)|^3&  \xrightarrow[]{L^1}& 0. \label{p3}
\end{eqnarray}
Indeed, we will have, since $|\ln (1+y)-y+y^2/2|\le C|y|^3$ for $|y|\le
1/2$,
\[\left|Z_{n,0}-\sum_{i=1}^nY_{i,n}(u)+\frac12
  \sum_{i=1}^nY_{i,n}(u)^2\right|\le C  \sum_{i=1}^n|Y_{i,n}(u)|^3\]
with probability going to one with $n$, so that
\begin{align*}
  Z_{n,0}-uZ_n+\frac{u^2}{2}\Gamma = \sum_{i=1}^nY_{i,n}(u)
  -uZ_n&+\frac{1}{2}[u^2\Gamma -\sum_{i=1}^nY_{i,n}(u)^2]\\
&+ Z_{n,0}-\sum_{i=1}^nY_{i,n}(u)+\frac12 \sum_{i=1}^nY_{i,n}(u)^2
\end{align*}
will tend to $0$ in probability if \eqref{p1}, \eqref{p2} and \eqref{p3} hold.

To prove \eqref{p1}, note that from \eqref{Yinubis} and \eqref{Zn}
\[\sum_{i=1}^nY_{i,n}(u) -uZ_n =
\pi_{m_0}\left(\sum_{i=1}^n R_{i,n}(u)-\sum_{i=1}^n R_{i,n}(0)\right),\]
and the equalities
\[\Eno\left|\sum_{i=1}^n R_{i,n}(u)\right|^2=\sum_{i=1}^n\Eno R_{i,n}(u)^2= n\Eno|R_{1,n}(u)|^2\]
will give the desired $L^2$-convergence  if we can prove that for each $u$,
\begin{equation}
  \label{cvR1nu}
  n\Eno|R_{1,n}(u)|^2\xrightarrow[n\to\infty]{} 0.
\end{equation}
To this end, we  look at the
expression \eqref{Rnxu} of $R_{n}(x,u)$ for fixed $u$. We have
$|\theta _{j,n}(u) - \theta|^{2d - 1} \le H(u)^{2d-1} n^{-1/2}$ for any
$\theta$ in the integrand, any $j$ and $n$,. We may thus write 
\begin{align*}
\left|R_n(x,u)\right| & \le \sum_{j=m_0}^m \pi_j(u) \int_{\theta _{m_0} - H(u) n^{-\frac{1}{4d-2}}}^{\theta _{m_0} + H(u)  n^{-\frac{1}{4d-2}}} \left\lvert f^{(2d)}(x, \theta) \right\rvert \frac{H(u)^{2d-1}n^{-1/2}}{(2d-1)!} \mathrm{d}\theta \\
                      & \lec_u n^{-1/2}  \int_{\theta _{m_0} - H(u) n^{-\frac{1}{4d-2}}}^{\theta _{m_0} + H(u)  n^{-\frac{1}{4d-2}}}  \left\lvert f^{(2d)}(x, \theta) \right\rvert \mathrm{d}\theta.
\end{align*}
Since we have $\sigma $-finite
measures, we may use Fubini theorem. Since moreover $\theta$ in the
integrand is between $\theta _0$ and $\theta _{j,n}(u)$ which
converges to $\theta _0$, we may then apply
Proposition~\ref{tversmix}. For $q\in \lb 1,4\rb$, using convexity of $x \mapsto x^q$ on line two, we may then write:
\begin{align*}
    \Eno\left|R_{1,n}(u)\right|^q    & \lec_u   n^{-q/2}  \Eno\left|\frac{  \int_{|\theta-\theta _{m_0} |\lec_u n^{-\frac{1}{4d-2}}}   \left\lvert f^{(2d)}(x, \theta) \right\rvert \dd\theta }{f(x, G_n(0))}\right|^q \\
      & \lec_u   n^{-\frac{q}{2}-\frac{q-1}{4d-2}}   \int_{|\theta-\theta _{m_0}|\lec_u  n^{-\frac{1}{4d-2}}} \Eno\left|\frac{f^{(2d)}(x, \theta)  }{f(x, G_n(0))}\right|^q \dd\theta \\
      & \lec_u   n^{-\frac{q}{2}-\frac{q}{4d-2}}  C \\
      & \lec_u n^{-\frac{q}{2}-\frac{q}{4d-2}}  
\end{align*}
with $C$ from Proposition \ref{tversmix}. In particular,
\begin{equation}
  \label{cvEn0}
 n^{q/2} \Eno\left|R_{1,n}(u)\right|^q \lec_u n^{-q/(4d-2)}  \to 0. 
\end{equation}
Take $q=2$ to obtain \eqref{cvR1nu}  ; the proof of \eqref{p1} is complete.

To prove \eqref{p2}, note first that from \eqref{Yinubis}  and \eqref{Zn},
\begin{align*}
 \sum_{i=1}^nY_{i,n}(u)^2-\frac{u^2\pi_{m_0}^2}{n}\sum_{i=1}^nZ_{i,n}^2=\pi_{m_0}^2&\sum_{i=1}^n(R_{i,n}(u)-R_{i,n}(0))^2\\
&+\frac{2u \pi_{m_0}^2}{\sqrt{n}}\sum_{i=1}^n(R_{i,n}(u)-R_{i,n}(0))Z_{i,n}
\end{align*}
so that taking the $L^1$-norm and by the Cauchy-Schwarz inequality,
\begin{multline*}
  \Eno\left|
\sum_{i=1}^nY_{i,n}(u)^2-\frac{u^2\pi_{m_0}^2}{n}\sum_{i=1}^nZ_{i,n}^2\right|\lec_u
                                                                      n\Eno
                                                              |R_{1,n}(u)|^2+n\Eno
                                                              |R_{1,n}(0)|^2\\
+\sqrt{n\Eno |R_{1,n}(u)|^2+n\Eno |R_{1,n}(0)|^2}\sqrt{\Eno Z_{1,n}^2}
\end{multline*}
and the r.h.s. tends to $0$ by \eqref{cvR1nu} and the fact that $\Eno Z_{1,n}^2\to\sigma^2$.
Moreover, setting $\delta_n:=|\Eno Z_{1,n}^2-\sigma^2|$, we have 
\begin{eqnarray*}
  \Eno \left|n^{-1}\sum_{i=1}^nZ_{i,n}^2-\sigma^2\right|^2&\lec &\Eno
                                                                  \left|n^{-1}\sum_{i=1}^n(Z_{i,n}^2-\Eno
                                                                  Z_{1,n}^2)\right|^2+\delta_n^2\\
&\lec &
                                                                  n^{-1}\mathrm{Var}_{n,0}(Z_{1,n}^2)+\delta_n^2\to
  0\\
\end{eqnarray*}
which goes to zero since $\delta_n\to 0$ by definition and $\Eno
Z_{1,n}^4\le C$ for some constant $C$ by Proposition~\ref{tversmix}.
We have thus,   
\[\frac1n\sum_{i=1}^nZ_{i,n}^2\xrightarrow[]{L^2} \sigma^2\quad\text{and}\quad\left|\sum_{i=1}^nY_{i,n}(u)^2-u^2\pi_{m_0}^2\frac{1}{n}\sum_{i=1}^nZ_{i,n}^2\right|\xrightarrow[]{L^1} 0\]
which prove \eqref{p2}.

We turn to the proof of \eqref{p3}. It is easily seen from \eqref{Yinubis} that
\begin{equation*}
 \sum_{i=1}^n|Y_{i,n}(u)|^3\lec_u
  n^{-3/2}\sum_{i=1}^n|Z_{i,n}|^3+\sum_{i=1}^n|R_{i,n}(u)|^3+\sum_{i=1}^n|R_{i,n}(0)|^3
\end{equation*}
so that taking expectations
\begin{equation*}
\Eno \sum_{i=1}^n|Y_{i,n}(u)|^3\lec_u
  n^{-1/2}\Eno|Z_{1,n}|^3+n\Eno |R_{1,n}(u)|^3+n\Eno|R_{1,n}(0)|^3.
\end{equation*}
But each of the three terms in the r.h.s. tends to $0$: the first one
because of $\Eno|Z_{1,n}|^3\le C$ by Proposition~\ref{tversmix}, the
second and the third ones because of  \eqref{cvEn0} for $q=3$. Thus
$\sum_{i=1}^n|Y_{i,n}(u)|^3$ converges to $0$ in $L^1$.
\end{proof}

\begin{ex}
Let's take $m=2$, $m_0=1$ and $\theta_{m_0}=0$ so that $G_0 = \delta _0$. Then $G_{1,n} = \frac{1}{2} \left( \delta
      _{- 2 n^{-1/6}} + \delta _{2n^{-1/6}} \right) $ and $G_{2,n} =
    \frac{4}{5} \delta _{-n^{-1/6}} + \frac{1}{5} \delta _{4
      n^{-1/6}}$ both have $0$ as first moment, and $4n^{-1/3}$ as
    second moment. The third moments are respectively zero for $G_{1,n}$ and $12n^{-1/2}$ for $G_{2,n}$. With the notation \eqref{mixdist} in the proof of Theorem~\ref{LAN}, 
we have $G_{1,n}=G_n(0)$ and $G_{2,n}=G_n(12)$. Clearly, one has $\W(G_{1,n} ,G_{2,n})= n^{-1/6} $ for all $n$
and as a by-product of Theorem~\ref{LAN} \eqref{lan3},  $\{G_{1,n}\}$ and $\{G_{2,n}\} $ are contiguous.
\end{ex}




\section{The rate $n^{-1/(4(m-m_0) + 2)}$ is optimal}
\label{sec:upperbound}

We follow \citet{Deely&Kruse} and \citeauthor{Chen}'s \citeyearpar{Chen} strategy of estimating $G$ by 
minimizing the $L^{\infty}$ distance to the empirical repartition function \eqref{gn}. We then need to control this distance in terms of the Wasserstein metric (Theorem \ref{orders}), under appropriate identifiability conditions. To do so, we consider sequences of couples $(G_{1,n}, G_{2,n})$ minimizing the relevant ratios, and express $F(x, G_{1,n}) - F(x,G_{2,n})$ as a sum on their components $F(x, \theta_{j,n})$ and relevant derivatives. A difficulty arises: distinct components $\theta_{j,n}$ may converge to the same $\theta_j$, leading to cancellations in the sums. Forgetting this case was the mistake by \citet{Chen} in the proof of their Lemma 2. We deal with it by using a coarse-graining tree: each node corresponds to sets of components that converge to the same point at a given rate. We may then use Taylor expansions on each node and its descendants, while ensuring that we keep non-zero terms (Lemma \ref{lemrec}). 

\subsection{Strong identifiability of order $k$}
In what follows $\|\cdot\|_\infty$ is the supremum norm with respect to $x$ and $\|\cdot\|$ is the Euclidean norm (for instance). Recall that 
$ F^{(p)}(x, \theta )$ is the $p$-derivative of $F(x,\theta)$ with respect to $\theta$.
\begin{defin}
    \label{identifiability}
    A family $\left\{ F(x,\theta ), \theta \in \Theta  \right\} $ of
    distribution functions is \emph{$k$-strongly identifiable} if for any finite set of say $m$ distinct $\theta_j$, then the equality
    \begin{align*}
        \norm{ \sum_{p=0}^k \sum_{j=1}^m \alpha _{p,j} F^{(p)}(x, \theta _j) }_{\infty} = 0
    \end{align*}
    implies $\alpha _{p,j} = 0$ for all $p$ and $j$.
\end{defin}

\begin{rem}
    \label{infimum}
    For a $k$-strongly identifiable family and fixed $\theta _i$, we may consider 
    \begin{align*}
        \inf_{\norm{\alpha} = 1}  \norm{ \sum_{p=0}^k \sum_{j=1}^m \alpha _{p,j} F^{(p)}(x, \theta _j) }_{\infty}. 
    \end{align*}
    Since the inner norm is a continuous function of $\alpha$ and the
    sphere is compact, this infimum is attained, and hence not
    zero:  for some  $c(\theta_1,\ldots,\theta_m) > 0 $, we have:
    \begin{align}
        \label{alpha_bound}
        \norm{ \sum_{p=0}^k \sum_{j=1}^m \alpha _{p,j} F^{(p)}(x, \theta _j) }_{\infty} \geq c(\theta_1,\ldots,\theta_m) \norm{\alpha }.
    \end{align}   
\end{rem}
\subsection{Main result and corollaries}
\begin{thm}
    \label{orders}
   Assume that  $\left\{ F(x, \theta ), \theta \in \Theta  \right\} $ is
    $2m$-strongly identifiable and that $F(x, \theta )$ is $2m$-differentiable with respect to
   $\theta $ for all $x$, with
    \begin{equation}\label{po}        
        F^{(2m)}(x,\theta _1) - F^{(2m)}(x,\theta _2) = o(\theta _1 - \theta _2)
    \end{equation}
    uniformly in $x$. Then, for any $G_0\in \mathcal{G}_{m_0}$, there are $\varepsilon >0$ and $\delta>0 $ such that 
\begin{align}
    \label{local}
    \inf_{\substack{G_1,G_2\in\Glm \\G_1\ne G_2\\ \W(G_1,G_0) \vee \W(G_2,G_0)\le \eps} }   \frac{\left\lVert F(x, G_1) - F(x, G_2) \right\rVert _{\infty}}{\W(G_1, G_2)^{2m - 2m_0 + 1}} > \delta. 
\end{align}
\end{thm}

\begin{cor}
\label{global}
  Under the conditions of Theorem~\ref{orders}, there exists $\delta > 0$ such that
\begin{equation}
    \label{general}
  \inf_{\substack{G_1,G_2\in\Glm\\G_1\ne G_2}} \frac{\left\lVert F(x, G_1) - F(x, G_2) \right\rVert _{\infty}}{\W(G_1, G_2)^{2m-1}}  > \delta. 
\end{equation}
\end{cor}

\begin{proof}[Proof of Corollary~\ref{global}]
Consider a sequence
  $(G_{1,n},G_{2,n})$ in $\Glm^2$ with $G_{1,n}\ne G_{2,n} $ for
  each $n$ and such that 
  \begin{equation}
\label{mi}
   \frac{\left\lVert F(x, G_{1,n}) - F(x, G_{2,n}) \right\rVert _{\infty}}{\W(G_{1,n}, G_{2,n})^{2m-1}}\xrightarrow[n\to\infty]{}\inf_{\substack{G_1,G_2\in\Glm\\G_1\ne G_2}} \frac{\left\lVert F(x, G_1) - F(x, G_2) \right\rVert _{\infty}}{\W(G_1, G_2)^{2m-1}} .
  \end{equation}
We can assume that 
$(G_{1,n},G_{2,n})$ converges to some limit $(G_{1,\infty},G_{2,\infty})$ in the compact set
$\Glm^2$. Distinguish two cases. 

Suppose first that $G_{1,\infty}\ne  G_{2,\infty}$.  Set
$w:=\W(G_{1,\infty}, G_{2,\infty})>0$ and  let $x_0$ such that
$z_0:=|F(x_0,G_{1,\infty})-F(x_0,G_{2,\infty})|>0$. Then, for all $n$
\begin{equation}
  \label{eq:1}
 \frac{\left\| F(x, G_{1,n}) - F(x, G_{2,n}) \right\|_{\infty}}{\W(G_{1,n}, G_{2,n})^{2m-1}}\ge \frac{\left| F(x_0, G_{1,n})- F(x_0, G_{2,n}) \right|}{\W(G_{1,n}, G_{2,n})^{2m-1}}.
\end{equation}
The numerator of the r.h.s. of \eqref{eq:1} tends to $z_0$ since $| F(x_0, G_{i,n})- F(x_0,
G_{i,\infty}) |$ is bounded by $K_0 \W(G_{i,n},G_{i,\infty})$ with $K_0=\max_{\theta\in
  \Theta}|F^{(1)}(x_0,\theta)|$  ($i=1,2$). And by
assumption, $\W(G_{1,n}, G_{2,n})$ tends to $w$.  As a
consequence,  \eqref{eq:1} and \eqref{mi} give \eqref{general}  by choosing $\delta:=z_0/w^{2m-1}$.
  
Suppose now that $G_{1,\infty}= G_{2,\infty}$. Set $G_0:= G_{1,\infty}$
which is in $\mathcal{G}_{m_0}$ with some  $m_0$ at most $m$. Consider $\eps
>0$ and $\delta>0$ as defined in \eqref{local} ; for $n$ large enough,
say $n\ge n_0$, $\W(G_{i,n},G_0)$ ($i=1,2$) is less than $\eps$ so that by \eqref{local}
\[\inf_{n\ge n_0}\frac{\left\lVert F(x, G_{1,n}) - F(x, G_{2,n}) \right\rVert _{\infty}}{\W(G_{1,n}, G_{2,n})^{2m -2m_0+1}}>\delta.\]
Moreover, for $n$ large enough, say $n \ge n_1$, $ \W(G_{1,n}, G_{2,n})$
is small so that 
$\W(G_{1,n}, G_{2,n})^{2m -2m_0+1}$ is more than $ \W(G_{1,n},G_{2,n})^{2m-1}$
and thus for all $n\ge n_0+n_1$,
\begin{equation*}
 \frac{\left\lVert F(x, G_{1,n}) - F(x, G_{2,n})
   \right\rVert _{\infty}}{\W(G_{1,n}, G_{2,n})^{2m-1}}\ge  \inf_{n\ge n_0+n_1}\frac{\left\lVert F(x, G_{1,n}) - F(x, G_{2,n})\right\rVert _{\infty}}{\W(G_{1,n}, G_{2,n})^{2m -2m_0+1}}> \delta.
\end{equation*}
\end{proof}

\begin{cor}
    \label{main}
Let $\eps >0$. Under the assumptions of Theorem~\ref{orders}, let $G_0\in\Gm0$ and $F_n$ be the empirical distribution of $n$ i.i.d. random variables with distribution $F(x,G_1)$. Let $\widehat{G}_n$ be a near optimal estimator of $G_1$ in the following sense: 
\begin{equation}
  \label{gn}
\|F(x,\widehat{G}_n)-F_n(x)\|_\infty\le \inf_{G\in \Glm}\|F(x,G)-F_n(x)\|_\infty+\frac{1}{n}.  
\end{equation}
Then, 
\[W(\widehat{G}_n,G_1)\lec \frac{1}{n^{1/(4(m-m_0)+2)}}\]
in probability under $G_1$, uniformly for $G_1\in \Glm$ such that $\W(G_1,G_0)<\eps$.
\end{cor}
\begin{proof}[Proof of Corollary~\ref{main}] We simply follow 
  \citet[Theorem 2]{Chen}. By the triangle inequality and \eqref{gn} (choose $G=G_1$), we have 
\[\|F(x,\widehat{G}_n)-F(x,G_1)\|_\infty\le 2\|F(x,G_1)-F_n(x)\|_\infty+\frac{1}{n}.\]
Moreover by the DKW inequality \citep{Ma}, we have 
\[\|F(x,G_1)-F_n(x)\|_\infty\lec \frac{1}{\sqrt{n}}, \] 
and thus 
\begin{equation}
  \label{cvpb}
\|F(x,\widehat{G}_n)-F(x,G_1)\|_\infty\lec \frac{1}{\sqrt{n}}   
\end{equation}
in probability under $G_1$, uniformly in $G_1$.

We also have $\W(\widehat{G}_n,G_1)\to 0$. Otherwise, since $\widehat{G}_n$ is in the compact space $\Glm$, there would be a subsequence $\widehat{G}_{n_k}$ which converges to some $G_2\ne G_1$ and thus we would have for all $x$:

\[|F(x, \widehat{G}_{n_k})-F(x,G_2)|\le \max_{\theta\in\Theta}|F^{(1)}(x,\theta)|\, \W(\widehat{G}_{n_k},G_2)\to 0.\]
This, together with \eqref{cvpb}, would imply $|F(x,G_1)-F(x,G_2)|=0$ for all $x$, which contradicts identifiability. 

Consequently, if $\W(G_1,G_0)<\eps$,  we have  $\W(\widehat{G}_n,G_0)<2\eps$ for $n$ large enough, and by Theorem~\ref{orders} and \eqref{cvpb},
  \[\W(\widehat{G}_n,G_1)^{2m-2m_0+1}\lec \|F(x,\widehat{G}_n)-F(x,G_1)\|_\infty\lec \frac{1}{\sqrt{n}}\]
in probability under $G_1$, uniformly in $G_1\in \Glm$ such that $\W(G_1,G_0)<\eps$. 
\end{proof}

\subsection{Proof of  the main Theorem~\ref{orders}}
In all this section, keep in mind the hypothesis of Theorem~\ref{orders}: the family $\left\{ F(x, \theta ), \theta \in \Theta  \right\} $ is
    $2m$-strongly identifiable and $F(x, \theta )$ is $2m$-differentiable with respect to
   $\theta $ for all $x$, with
    \begin{equation*}       
        F^{(2m)}(x,\theta _1) - F^{(2m)}(x,\theta _2) = o(\theta _1 - \theta _2)
    \end{equation*}
    uniformly in $x$. 
Note first that proving \eqref{local} amounts to proving 
\[\lim_{n\to \infty}\uparrow \inf_{\substack{G_1,G_2\in\Glm \\G_1\ne G_2\\ \W(G_1,G_0) \vee \W(G_2,G_0)\le 1/n} }   \frac{\left\lVert F(x, G_1) - F(x, G_2) \right\rVert _{\infty}}{\W(G_1, G_2)^{2m - 2m_0 + 1}} > \delta.\]
From now on, we  consider two sequences  $(G_{1,n}), (G_{2,n})$ in
$\Glm$  such that for each $n\ge 1$:
\begin{itemize}
\item $G_{1,n}\ne G_{2,n}$,
\item $\W(G_{i,n},G_0)\le\frac{1}{n}$ ($i=1,2$),
\item \[ \inf_{\substack{G_i\in\Glm \\G_1\ne G_2\\ \W(G_i,G_0)\le \frac{1}{n}} } \!\!\!  \frac{\left\| F(x, G_1) - F(x, G_2) \right\| _{\infty}}{\W(G_1, G_2)^{2m - 2m_0 + 1}}\ge \frac{\left\| F(x, G_{1,n}) - F(x, G_{2,n}) \right\| _{\infty}}{\W(G_{1,n}, G_{2,n})^{2m - 2m_0 + 1}}-\frac{1}{n}.\]
\end{itemize}
Consequently, it's enough to prove that 
\begin{equation}
  \label{localn}
\liminf_{n\to\infty }\frac{\left\| F(x, G_{1,n}) - F(x, G_{2,n}) \right\| _{\infty}}{\W(G_{1,n}, G_{2,n})^{2m - 2m_0 + 1}}>\delta.
\end{equation}
Since $(G_{1,n}), (G_{2,n})$ are two sequences in  $\Glm$ and $m$ is
finite, we may and  do assume that $(G_{i,n})\subset \mathcal{G}
_{m_{i}}$ for some $m_{i} \le m$  and $i=1,2$.  We can then write for
each $n$
\[G_{1,n} =\sum_{j=1}^{m_1}\pi_{1,j,n}\delta_{\theta_{1,j,n}}
   \quad\text{and}\quad G_{2,n} =\sum_{j=m_1+1}^{m_1+m_2}\pi_{2,j,n}\delta_{\theta_{2,j,n}}\]
and define for
each $n$ a signed measure $G_n$ of total mass zero:
\begin{equation*}
  G_n=G_{1,n}- G_{2,n} =\sum_{j=1}^{m_1+m_2}\pi_{j,n}\delta_{\theta_{j,n}}
\end{equation*}
with 
\[(\pi_{j,n} ,\theta_{j,n})=
\begin{cases}
(\pi_{1,j,n},\theta_{1,j,n}) & \text{for $j\in\lb 1,m_1\rb$ }\\
(- \pi_{2,j,n}, \theta_{2,j,n}) & \text{for $j\in\lb m_1+1,m_2\rb$ }
\end{cases}.
\]




\subsubsection{Scaling sequences}
\label{sec:scale}
Set for short
\[J_o=\lb 1,m_1+m_2\rb.\] 
Since $J_o$ is finite, up to selecting a subsequence of $G_n$, we may find a finite number of scaling sequences $\eps_{0,n},\eps_{1,n},\ldots,\eps_{\sm,n}$, 
together  with  integers  $\fs(j,k)$  and $\vs(J)$  in $\lb0,\sm\rb$  for any $j,k \in J_o$ and $J\subset J_o$,
such
that
\begin{align}
    0\equiv \eps_{0,n} < \eps_{1,n} & < \cdots < \eps_{\sm,n}  \equiv  1, & \text{with }\eps_{s,n} & = o\big(\eps_{s+1,n}\big) , \notag \\
\label{sjk}
 \left\lvert \theta_{j,n} - \theta_{k,n}
 \right\rvert  & \asymp \eps_{\fs(j,k),n}, && \\
\label{sJ}
 \left\lvert \sum_{j\in J}\pi_{j,n}
 \right\rvert &  \asymp \eps_{\vs(J),n}. &&
\end{align}
We also define  the $\fs$-diameter of $J$ as
\begin{align*}
    \fs(J) & = \sup_{j,k \in J} \fs(j,k).   
\end{align*}

\subsubsection{Defining a tree for the key lemmas}
\label{sec:tree}
Note that the application $\fs(\cdot,\cdot)$ defined by \eqref{sjk} is an ultrametric on $J_o$ (but does not
separate points). Thus we may define a tree $\T$ whose vertices are indexed by the distinct ultrametric closed balls
$J=B_\fs(j,s)$  when $j$ ranges over $J_o$ and $s$ over
$\lb0,\sm\rb$.

Indeed, if $I$ and $J$ are two such balls, and $I \cap J\ne \emptyset$, then  either  $I\subset J$ or $J\subset I$.
 
So that, defining the set of descendants and the set of children of $J$ by
\begin{eqnarray*}
 \mathrm{Desc}(J) &=& \{I\in\T : I\subsetneq J\}; \\
\mathrm{Child}(J) & =& \{I\in\D{J} : I\subset H\subsetneq J,H\in\T\Longrightarrow
H=I\},
\end{eqnarray*}
we get a tree $\T$ with root $J_o$, and where the parent of  $J\ne J_o$ is given by 
\[p(J)=K\iff J\in\C{K}.\]

\begin{lem}\label{scW} With the above notations, given the tree $\T$,
  \begin{equation}
\label{scaleW}
    \W\left(G_{1,n}, G_{2,n}\right)  \asymp\max_{J\in \mathrm{Desc}(J_o)} \eps_{\vs(J),n} \eps_{\fs(p(J)),n}.
\end{equation}
\end{lem}
\begin{proof}
  See Appendix~\ref{apscW}.
\end{proof}
Set now 
\[F(x,G_n):=F(x,G_{1,n})-F(x,G_{2,n})\]
and for $J\subset J_o$, 
\[F(x,J):=\sum_{j\in J}\pi_{j,n}F(x,\theta_{j,n}).\]
Note that $F(x,G_n)=F(x,J_o)$. 
We now use  Taylor expansions along the tree $\T$ to express the order of $F(x,G_n)$ in terms of the scaling functions $\eps_s$.

\begin{lem}\label{lemrec}
  Let $J$ be a vertex of the tree $\T$ and set $d_J=\mathrm{card}(J)$. Pick $\theta_J:=\theta_{J,n}$ in the set $\{\theta_{j,n}:j\in J\}$. The subscript $_n$ is  skipped from the following notations. There is a vector $\eta_J=(\eta_{k,J})_{0\le k\le 2m}$ and a remainder $R(x,J)$ such that 
\begin{equation}
    \label{hyprec}
    F(x,J) = \sum_{k=0}^{2m} \eta_{k,J} \eps_{\fs(J)}^k F^{(k)}(x, \theta _J) + R(x,J),
\end{equation}
where:
\begin{enumerate}[(i)]
\item \label{coeffsbornes}$\displaystyle\eta_{0,J}=\sum_{j\in
    J}\pi_j$  and $|\eta_{k,J}|\lec 1 $ for all $k\le 2m$;
 \item \label{premierscoeffs} Taking subsequences if needed, there is a coefficient
   $\eta_{k,J}$ of maximal order among the $d_J$ first
   ones. That is, there is an integer $k(J)<d_J$ such that 
            \begin{equation*}
            \|\eta_J \|:= \max_{k\le 2m} | \eta_{k,J} |
            \asymp  |\eta_{k(J),J}| ;
        \end{equation*}
\item  \label{borneinfcoeff}  The norm $\|\eta_J \|$ is bounded from below (up to a constant) by
  a quantity linked to the Wasserstein distance: 
\[\|\eta_J \|\gec \max\left(\eps _{\vs(J)},
  \max_{I\in\mathrm{Desc}(J)}\eps _{\vs(I)} \left(\frac{\eps
      _{\fs(p(I))}}{\eps_{\fs(J)}}\right) ^{d_J - 1}\right);\]
\item \label{remaind}  The remainder term is negligible. Uniformly in $x$:
 \[R(x,J) = o\left(\|\eta_J\|\, \varepsilon_{\fs(J)}^{2m}\right).\]
\end{enumerate}  
\end{lem}
\begin{proof}
  See Appendix~\ref{aplemrec}.
\end{proof}

\subsubsection{Concluding the proof}
Let us now consider the root $J_o$ of the tree $\T$. Distinguish two cases:
\begin{description}
\item[Case 1.] Assume that $\fs(J_o)<\sm$. We have $\eps_{\fs(J_o)}=o(1)$ and may apply directly Lemma \ref{hyprec} to $J_o$:
\[F(x,G_n) = F(x,J_o)=\sum_{k=0}^{2m} \eta_{k,J_o} \eps _{\fs(J_o)}^k
F^{(k)}(x, \theta _{J_o}) + R(x,J_o) , \]
where at least one $\eta_{k,J_o}$ satisfies 
\begin{eqnarray*}
 |\eta_{k,J_o}|&\gec&
 \max_{I\in\D{J_o}}\eps_{\vs(I)}\left(\frac{\eps_{\fs(p(I))}}{\eps_{\fs(J_o)}}\right)^{d_{J_o}-1}\gec \max_{I\in\D{J_o}}\eps_{\vs(I)}\left(\frac{\eps_{\fs(p(I))}}{\eps_{\fs(J_o)}}\right)^{2m-1} ,
\end{eqnarray*}
so that one of the coefficients of the derivatives satisfies 
\[|\eta_{k,J_o}\eps _{\fs(J_o)}^k|\gec \max_{I\in\D{J_o}}\eps_{\vs(I)}
\eps_{\fs(p(I))}^{2m-1}.\]
Thus, taking $i=1$ in the lower bound  \eqref{alpha_bound}, and since $R(x,J_o)$ is
of smaller order, we get 
\[\left\lVert F(x,G_n) \right\rVert _{\infty}\gec  \max_{I\in\D{J_o}}\eps_{\vs(I)}
\eps_{\fs(p(I))}^{2m-1}\gec \W(G_{1,n},G_{2,n})^ {2m-1} , \]
where the last inequality comes from Lemma~\ref{scW}.
\item[Case 2.]  Assume that $\fs(J_o)=\sm$. We split $G_n$ over the
  first-generation children:
  \begin{eqnarray*}
  F(x,G_n) = F(x,J_o)&=&\sum_{I\in\C{J_o}}F(x,I)\\
&=& \sum_{I\in\C{J_o}}\left[\sum_{k=0}^{2m} \eta_{k,I} \eps _{\fs(I)}^kF^{(k)}(x, \theta _{I}) + R(x,I)  \right].
  \end{eqnarray*}
Moreover the $\theta _{I}$ for
$I\in\C{J_o}$ are $\eps$-separated for some $\eps >0$ (see \eqref{sep}), so that the
lower bound  \eqref{alpha_bound} can be applied and yields, since the
$R(x,I)$'s are negligible:
\begin{equation*}
 \left\| F(x,G_n) \right\| _{\infty}\gec
  \max_{I\in\C{J_o}}\max_{k\le 2m}|\eta_{k,I}\eps _{\fs(I)}^k|\ge
  \max_{I\in\C{J_o}}\max_{k<d_I}|\eta_{k,I}\eps _{\fs(I)}^k|. 
\end{equation*}
On the one hand, we have $\max_{k<d_I}|\eta_{k,I}\eps _{\fs(I)}^k|\ge |\eta_{0,I}|$ and since $|\eta_{0,I}|=|\sum_{j\in I}\pi_j|\asymp \eps_{\vs(I)}$, we deduce 
\[\left\lVert F(x,G_n) \right\rVert _{\infty}\gec
\max_{I\in\C{J_o}}\eps_{\vs(I)}.\] 
On the other hand, we have $\max_{k<d_I}|\eta_{k,I}\eps _{\fs(I)}^k|\ge \max_{k<d_I}|\eta_{k,I}|\eps _{\fs(I)}^{d_I-1}$ so that from Lemma~\ref{lemrec} \eqref{premierscoeffs} and \eqref{borneinfcoeff} for $I$, we deduce further 
\begin{eqnarray*}
  \left\lVert F(x,G_n) \right\rVert _{\infty} &\gec&
  \max_{I\in\C{J_o}}\|\eta_I\|\eps _{\fs(I)}^{d_I-1}\\
&\gec&
  \max_{I\in\C{J_o}}\max_{H\in\D{I}}\eps_{\vs(H)}
\eps_{\fs(p(H))}^{d_I-1}.
\end{eqnarray*}
After recalling that $\eps_{\fs(J_o)}=1$ and setting $d_\star=\max_{I\in\C{J_o}}d_I$, we may combine these two lower bounds and get
\begin{eqnarray*}
  \left\lVert F(x,G_n) \right\rVert _{\infty}&\gec&  \max_{I\in\C{J_o}}\max_{H\in\D{I}\cup\{I\}}\eps_{\vs(H)}
\eps_{\fs(p(H))}^{d_I-1}\\
&\gec&
\max_{H\in\D{J_o}}\eps_{\vs(H)}\eps_{\fs(p(H))}^{d_\star-1}\\
&\gec& \W(G_{1,n},G_{2,n})^ {d_\star-1} ,
\end{eqnarray*}
where the last inequality comes from Lemma~\ref{scW}. Since $G_{1,n}$ and $G_{2,n}$ converge to $G_0\in \Gm0$, the root $J_o$ (of cardinality $m_1+m_2$) has at least $m_0$ children with at least two elements. Thus, the cardinality $d_\star$ of the biggest child is bounded by $m_1+m_2-2(m_0-1)$. Thus, 
\[ \left\lVert F(x,G_n) \right\rVert _{\infty}\gec \W(G_{1,n},G_{2,n})^ {m_1+m_2-2m_0+1}\gec \W(G_{1,n},G_{2,n})^ {2m-2m_0+1}.\]
\end{description}
Finally, if $m_0$ is more than one, we are in the second case (where $\fs(J_o)=\sm$) and if $m_0$ is one, the two cases can occur. But whatever the case, we always have 
\[\left\lVert F(x,G_n) \right\rVert _{\infty}\gec \W(G_{1,n},G_{2,n})^ {2m-2m_0+1}\]
so that \eqref{localn} is proved.

\section{A class of $k$-strongly identifiable families}
\label{sec:class}

We expect the strong identifiability to be rather generic, and hence the above theory often meaningful. In particular, \citet[Theorem~3]{Chen} has proved that 
location and scale families with smooth densities are $2$-strongly identifiable. The theorem and the proof straightforwardly generalise to our case. We merely state the result.
\begin{thm}
    \label{thm_identifiability}
Let $k\ge 1$. Let $f$ be a probability density with respect to to the Lebesgue measure. Assume that $f$ is $k-1$ times differentiable with 
\[\lim_{x\to\pm \infty}f^{(p)}(x)=0\text{ for } p\in \lb 0,k-1\rb.\]
Set $F(x,\theta)=\int_{-\infty}^x f(y-\theta)dy$. Then the family $\{F(x,\theta),\theta\in\Theta\}$ is $k$-strongly identifiable.  If $\Theta\subset (0,\infty)$, the result stays true with $F(x,\theta)=\frac{1}{\theta}\int_{-\infty}^x f\left(\frac{y}{\theta}\right)dy$.
\end{thm}

\appendix
\section{Auxiliary Proofs}
\label{app}
\subsection{Proof of Equation \eqref{Jac}}
\label{jaco}
The map 
\[\phi :(\pi_1,\ldots,\pi_d,\theta_1,\ldots,\theta_d)\mapsto 
\left(\sum_1^d\pi_j,\sum_1^d\pi_j\theta_j,\sum_1^d\pi_j\theta_j^2,\ldots,\sum_1^d\pi_j\theta_j^{2d-1}\right)\]
has the following Jacobian : 
\[
J(\phi)=(-1)^{\frac{(d-1)d}{2}}\,\pi_1\cdots\pi_d \prod_{1\le j<k\le d}(\theta_j-\theta_k)^4.  \]
To prove this, note that
\begin{eqnarray*}
 J(\phi) &= &
\begin{vmatrix}
 1 & \cdots & 1 & 0 & \cdots & 0\\
\theta_1 &  \cdots & \theta_d & \pi_1& \cdots & \pi_d\\
\theta_1^2 &  \cdots & \theta_d^2 & 2\pi_1\theta_1& \cdots & 2\pi_d\theta_d\\
\vdots &            & \vdots & \vdots &        & \vdots \\ 
\theta_1^{2d-1} &  \cdots & \theta_d^{2d-1} & (2d-1)\pi_1\theta_1^{2d-2}& \cdots & (2d-1)\pi_d\theta_d^{2d-2}
\end{vmatrix}\\
&=& \pi_1\cdots\pi_d \,\Delta_d
\end{eqnarray*}
with 
\begin{eqnarray*}
\Delta_d &=& \begin{vmatrix}
 1 & \cdots & 1 & 0 & \cdots & 0\\
\theta_1 &  \cdots & \theta_d & 1& \cdots & 1\\
\theta_1^2 &  \cdots & \theta_d^2 & 2\theta_1& \cdots & 2\theta_d\\
\vdots &            & \vdots & \vdots &        & \vdots \\ 
\theta_1^{2d-1} &  \cdots & \theta_d^{2d-1} & (2d-1)\theta_1^{2d-2}& \cdots & (2d-1)\theta_d^{2d-2}
\end{vmatrix}\\
& = & \begin{vmatrix}
 1 & \cdots & 1 & 0 & \cdots & 0\\
\theta_1 &  \cdots & \theta_d & 1& \cdots & 1\\
\theta_1^2 &  \cdots & \theta_d^2 & 2\theta_1& \cdots & 2\theta_d\\
\vdots &            & \vdots & \vdots &        & \vdots \\ 
\theta_1^{2d-2} &  \cdots & \theta_d^{2d-2} & (2d-2)\theta_1^{2d-3}& \cdots & (2d-2)\theta_{d-1}^{2d-3}\\
P(\theta_1) &  \cdots & P(\theta_d) & P'(\theta_1)& \cdots & P'(\theta_d)
\end{vmatrix},\\
\end{eqnarray*}
where $P$ can be any (normalized) polynomial of degree $2d-1$ and $P'$ its derivative. Choosing 
$P(\theta)=(\theta-\theta_d)\prod_{1\le j\le d-1}(\theta-\theta_j)^2$, we get 
\begin{eqnarray*}\Delta_d&=& P'(\theta_d) \begin{vmatrix}
 1 & \cdots & 1 & 0 & \cdots & 0\\
\theta_1 &  \cdots & \theta_d & 1& \cdots & 1\\
\theta_1^2 &  \cdots & \theta_d^2 & 2\theta_1& \cdots & 2\theta_{d-1}\\
\vdots &            & \vdots & \vdots &        & \vdots \\ 
\theta_1^{2d-2} &  \cdots & \theta_d^{2d-2} & (2d-2)\theta_1^{2d-3}& \cdots & (2d-2)\theta_{d-1}^{2d-3}
\end{vmatrix}\\
&=&  P'(\theta_d) \begin{vmatrix}
 1 & \cdots & 1 & 0 & \cdots & 0\\
\theta_1 &  \cdots & \theta_d & 1& \cdots & 1\\
\theta_1^2 &  \cdots & \theta_d^2 & 2\theta_1& \cdots & 2\theta_{d-1}\\
\vdots &            & \vdots & \vdots &        & \vdots \\ 
\theta_1^{2d-3} &  \cdots & \theta_{d-1}^{2d-3} & (2d-3)\theta_1^{2d-4}& \cdots & (2d-3)\theta_{d-1}^{2d-4}\\
Q(\theta_1) &  \cdots & Q(\theta_d) &Q'(\theta_1)& \cdots & Q'(\theta_{d-1})
\end{vmatrix},\\
\end{eqnarray*}
where $Q$ is any polynomial of degree $2d-2$. With
$\displaystyle Q(\theta)=\prod_{1\le j\le d-1}(\theta-\theta_j)^2$, we obtain
\begin{equation*}
\Delta_d=(-1)^{d-1} P'(\theta_d) Q(\theta_d)\Delta_{d-1}=(-1)^{d-1} \prod_{j=1}^{d-1}(\theta_d-\theta_j)^4 \Delta_{d-1}.
\end{equation*} 
By iteration, we get 
\begin{eqnarray*}
 \Delta_d &=&(-1)^{d-1} \prod_{j=1}^{d-1}(\theta_d-\theta_j)^4
(-1)^{d-2}
\prod_{j=1}^{d-2}(\theta_{d-1}-\theta_j)^4\Delta_{d-2}\\
&=&
(-1)^{d-1+d-2+\cdots+1}\prod_{k=2}^d\prod_{j=1}^{k-1}(\theta_k-\theta_j)^4
\Delta_1 \\
&=&
(-1)^{\frac{(d-1)d}{2}}\prod_{1\le j<k\le d}(\theta_k-\theta_j)^4
\end{eqnarray*}
since $\Delta_1=1$. The proof is complete.

\subsection{Auxiliary matrix tool}

\begin{lem}
    \label{indep}
    Suppose $j$, $d_i$ and $d$ are positive integers such that
    $\sum_{i=1}^j d_i = d$. Consider numbers $\theta_1,\cdots,\theta_j$ all
    distinct. Write 
\[\mathcal{I} =\left\{ (i,\ell) \in \mathbb{N} : 1
      \le i \le j, 1 \le \ell \le d_i\right\}.\]
 Define for each $(i,\ell)\in\mathcal{I}$ a 
    $d$-dimensional column vector as follows:
\[ a_{i,\ell}[k]  =\dfrac{\theta _i^{k-\ell}}{(k-\ell)!}\1_{k\ge \ell},\quad 
1\le k\le d, 
\]
and stack these vectors in a $d\times d$ matrix 
\begin{equation}
  \label{A}
A(\theta_1,\ldots,\theta_j)=\left[ a_{1,1} | \dots |a_{1,d_1}|\dots |a_{j,1}| \dots | a_{j,d_j} \right]. 
\end{equation}
Then, the rank of $A(\theta_1,\ldots,\theta_j)$ is $d$.
\end{lem}
\begin{proof}
 Set for short $A= A(\theta_1,\ldots,\theta_j)$. Let
$\Lambda=(\lambda_{i,\ell})_{(i,\ell)\in\mathcal{I}}  $ be a vector such
that $A \Lambda  = 0$. Proving the lemma is equivalent to proving that
$\Lambda = 0$. Note that 
\begin{equation}
  \label{AL}
  (A \Lambda)_k=\sum_{(i,\ell) \in \mathcal{I} } \lambda
  _{i,\ell}a_{i,\ell}[k]= \sum_{i=1}^j \sum_{\ell=1}^{d_i}\lambda
  _{i,\ell}\dfrac{\theta _i^{k-\ell}}{(k-\ell)!}\1_{k\ge \ell},
\end{equation}
and for any $(d-1)$-degree polynomial $P(x) = \sum_{k=0}^{d-1}
c_k\frac{x^k}{k!}$ , we have
\begin{equation}
  \label{eq:P}
 (c_0, \dots, c_{d-1}) A \Lambda = \sum_{k=0}^{d-1}c_k (A \Lambda)_{k+1}=\sum_{(i,\ell) \in \mathcal{I} } \lambda _{i,\ell} P^{(\ell-1)}(\theta _i).
\end{equation} 
Hence, if $A \Lambda = 0$, then \eqref{eq:P} is zero. In particular,
the $(d-1)$-degree polynomials 

\[P_k(x) = (x-\theta_k)^{d_k-1}\prod_{\substack{i=1\\i\ne k}}^j (x-\theta_i)^{d_i},\quad 1\le k\le j,\]
yield
\[\sum_{(i,\ell)\in \mathcal{I}} \lambda _{i,\ell}  P_k^{(\ell-1)}(\theta_i) = \lambda _{k, d_k} (d_k-1)!\prod_{\substack{i=1\\i\ne k}}^j (\theta_k-\theta_i)^{d_i}=0,\]
so that $\lambda _{k, d_k} = 0$. More generally,
\[P_{k,q} (x) =(x-\theta_k)^{d_k-q}\prod_{\substack{i=1\\i\ne k}}^j (x-\theta_i)^{d_i}\quad (1\le q\le d_k)\]
 has the property that, for $(i,\ell) \in \mathcal{I} $, $P_{k,q}
 ^{(\ell-1)}(\theta_i) $ is zero if $i \ne k$ or $\ell\le d_k-q $. On the
 other hand this term is never zero if $i=k$ and $\ell = d_k-q+1$. So
 that recurrence on $q$ yields $\lambda _{k, d_k-q+1} = 0$, and hence
 $\Lambda = 0$. 
\end{proof}

\begin{defin}
Let $\eps>0$. A vector $(\theta _i)_{1\le i\le j}$ in $\Theta^j$ is said
  $\eps$-separated if 
  \[\forall i\ne i',\quad |\theta_i-\theta_{i'}|\ge \eps.\]
\end{defin}

\begin{cor}
    \label{memenorme}
Let $\eps >0$.  There exist constants $c>0$ and $C>0$  such
that for any vector $\Lambda $ and any $\eps$-separated vector $(\theta _i)_{1\le i\le j}$,
    \[
        c \left\lVert \Lambda  \right\rVert \le  \left\lVert A(\theta_1,\ldots,\theta_j) \Lambda  \right\rVert \le C \left\lVert \Lambda  \right\rVert,
    \]
where $A(\theta_1,\ldots,\theta_j)$ is as in \eqref{A}.
\end{cor}
\begin{proof}
 Note first that the set
 $\mathcal{D}_\eps$  of all  $\eps$-separated family
 $\{\theta _i\}_{1\le i\le j}$ is compact in $\Theta^j$. Moreover the norm $\|A(\theta_1,\ldots,\theta_j) \Lambda\|$
 is a continuous function of $((\theta_1,\ldots,\theta_j),\Lambda)$ on the compact space $\mathcal{D}_\eps\times S(0,1)$ where $S(0,1)$ is the  $d$-dimensional unit sphere.  Its infimum is attained on
 $\mathcal{D}_\eps\times S(0,1)$, say at $\left((\theta^* _i)_{1\le i
   \le j},\Lambda^*\right)$ . Now, by Lemma \ref{indep}, $c:=\|A(\theta_1^*,\ldots,\theta_j^*) \Lambda^*\|$ is positive so that $ c \left\lVert
   \Lambda  \right\rVert \le  \left\lVert A(\theta_1,\ldots,\theta_j) \Lambda  \right\rVert $
 for every $\Lambda$ and every $(\theta _i)_{1\le i\le j}$ in $\mathcal{D}_\eps$ .

    Conversely, $C$ is easily bounded from above by the sum of the norms of the matrix entries, and all those are bounded since $\Theta $ is compact.
\end{proof}

\subsection{Proof of Lemma~\ref{scW}}
\label{apscW}

We shall estimate $\W(G_{1,n},G_{2,n})$ with the comparison
scale and  the tree $\T$. 
 Set  for any function $f$ on $\Theta$ and any $J\subset J_o$
\begin{equation}
  \label{eq:not}
  f(J)=\sum_{j\in J}\pi_{j,n} f\left(\theta_{j,n}\right).
\end{equation}
In what follows the
  subscript $n$ is fixed
  and thus skipped in the $\theta_j$'s, $\pi_j$'s and $\eps_s$'s. Recall that the collection of distinct ultrametric balls
  $J=\{k:\fs(k,j)\le s\}$ for $j$ varying in $J_o$ and $s$ in
  $\{0,\ldots,\sm\}$ form a tree $\T$.  For each distinct $J$, we picked an arbitrary
  $j\in J$ and set $\theta_J=\theta_j$. Set also for short 
\[\pi(J)=\sum_{j\in J}\pi_j.\]
 Let $f$ be $1$-Lipschitz on $\Theta$. We first prove by recurrence
  that for any vertex $J$ of the tree,
\begin{equation}
\label{eq:recfj}
 f(J)\lec \pi(J)f(\theta_J)+\max_{I\in\mathrm{Desc}(J)}\eps_{\vs(I)}\eps_{s(p(I))}.
\end{equation}
If $J$ has $\fs$-diameter zero, then $f(J)=\pi(J)f(\theta_J)$ and
\eqref{eq:recfj} is satisfied. Next, if $J$ has children $I$ that satisfy \eqref{eq:recfj}, we compute 
\begin{eqnarray*}
   f(J) &=& \sum_{I\in \C{J}}f(I)\\
   &\lec & \sum_{I\in \C{J}}\left[\pi(I)f(\theta_I)+\max_{H\in \D{I}}\eps_{\vs(H)}\eps_{\fs(p(H))}\right]\\
&\le & \pi(J)f(\theta_J) +\!\!\!\!\sum_{I\in
  \C{J}}\!\!\!|\pi(I)|\underbrace{|f(\theta_I)-f(\theta_J)|}_{\le
    |\theta_I-\theta_J|} +\!\!\!\max_{H\in
    \D{I}}\eps_{\vs(H)}\eps_{\fs(p(H))}.\\
\end{eqnarray*}
Since $|\pi(I)|$ is of order $\eps_{\vs(I)} $ and
$|\theta_I-\theta_J|$ is of order $\eps_{\fs(J)}$ we see that \eqref{eq:recfj} holds for $J$ and in particular for $J_o$
for which we have $\pi(J_o)=0$.

To prove the reverse inequality, let $J\subsetneq J_o$ such that $\eps_{\vs(J)}\eps_{\fs(p(J))}$ is
maximal. Set $\fe(J)=\min_{j\notin J}|\theta_{j}-\theta_J|$ which is
bigger than the $\fs$-diameter of $J$ and consider a $1$-Lipschitz
function $f$ on $\Theta$ such that 
$f(J)=0$ and $ f(J_o\setminus J)=|\pi(J)|[\fe(J)-\fs(J)]$, for instance
\[f(\theta)=-\mathrm{sgn}(\pi(J))\times\min\{\fe(J)-\fs(J),[|\theta-\theta_J|-\fs(J)]_+\}\] 
which satisfies 
$f(\theta_j)=-\mathrm{sgn}(\pi(J))[\fe(J)-\fs(J)]\1_{j\notin J}$. We
get  
\[f(J_o) = f(J_o\setminus J)+f(J)= |\pi(J)|[\fe(J)-\fs(J)]\]
and since  $|\pi(J)|$ is of order $\eps_{\vs(J)}$ and $\fe(J)$
is of order $\eps_{\fs(p(J))}$ at least, we deduce
\begin{equation*}
 f(J_o)\gec\max_{I\in
    \D{J_o}}\eps_{\vs(I)}\eps_{\fs(p(I))}.
\end{equation*}  
It remains to note that  (by the Kantorovich-Rubinstein Theorem)
\begin{equation*}
  \W\left(G_{1,n},G_{2,n}\right)=\sup_{\Vert f\Vert_{\text{Lip}}\le 1}\int_\Theta f(\theta)dG_n(\theta)=\sup_{\Vert f\Vert_{\text{Lip}}\le 1}f(J_o).
\end{equation*}

\subsection{Proof of Lemma~\ref{lemrec}}
\label{aplemrec}
Recall that we set $\pi(J)=\sum_{j\in J}\pi_j$. We use definition
\eqref{eq:not} for $f(\theta)=F(x,\theta)$. If $J$ satisfies
$\fs(J)=0$, then all the $\theta_j$ for $j\in J$ are equal and
$F(x,J)=\pi(J)F(x,\theta_J)$. In this case, the choice $\eta_{k,J}=\pi(J)\1_{\{k=0\}}$ and  $R(x,J)=0$ work. 

\emph{Assume now that lemma~\ref{lemrec} holds for any vertex $I$ with
  parent $J$ in the tree $\T$.}   We write a Taylor expansion to pass the
estimates of $I$ to the parent $J$. By assumption,
    \begin{equation}
      \label{eq:FxI}
       F(x,I)  =  \sum_{\ell=0}^{2m} \eta_{\ell,I} \eps_{\fs(I)}^\ell  F^{(\ell)}(x, \theta _I) + R(x,I). 
    \end{equation}
    Assuming without loss of generality that $\theta_J\le \theta_I$, we apply
 Taylor's formula with remainder to $F^{(\ell)}(x, \theta _I)$ at $\theta_J$ and obtain
 \[F^{(\ell)}(x, \theta_I)-\!\!\sum_{k=\ell}^{2m-1}\frac{(\theta_I-\theta_J)^{k-\ell}}{(k-\ell)!}
F^{(k)}(x,\theta_J) = \!\! \int_{\theta_J}^{\theta_I}\!\! \frac{(\theta_I-\xi)^{2m-1-\ell}}{(2m-1-\ell)!}F^{(2m)}(x, \xi)d\xi . \] 
So that 
\begin{align*}
F^{(\ell)}(x, \theta_I)-&\sum_{k=\ell}^{2m}\frac{(\theta_I-\theta_J)^{k-\ell}}{(k-\ell)!}
F^{(k)}(x,\theta_J)  \\
&= \int_{\theta_J}^{\theta_I} \frac{(\theta_I-\xi)^{2m-1-\ell}}{(2m-1-\ell)!}\left[F^{(2m)}(x, \xi)-F^{(2m)}(x,\theta_J)\right]d\xi\\ 
&=\frac{(\theta_I-\theta_J)^{2m-1-\ell}}{(2m-1-\ell)!}\, O\left(\sup_{\xi\in
    [\theta_J,\theta_I]}|F^{(2m)}(x,\xi)-F^{(2m)}(x,\theta_J)|\right)\\
&=o\left((\theta_I-\theta_J)^{2m-\ell}\right),
\end{align*}
where we used assumption \eqref{po} in the last equality. Setting now
\begin{equation}
\label{etalIhI}
 \alpha _{\ell,I} =  \eta_{\ell,I} \left( \frac{\eps_{\fs(I)}}{\eps_{\fs(J)}} \right)^\ell, \quad   h_I   =     \frac{\theta_I - \theta _J}{\eps_{\fs(J)}}, 
\end{equation}
we obtain
\begin{equation*}
F^{(\ell)}(x, \theta_I)=  \sum_{k=\ell}^{2m}\eps_{\fs(J)}^k\frac{h_I^{k-\ell}}{(k-\ell)!}F^{(k)}(x,\theta_J)+\eps_{\fs(J)}^{2m-\ell}o(1) ,
\end{equation*}
and substituting in \eqref{eq:FxI}, we get
\begin{align*}
  F(x,I) 
 = \sum_{k=0}^{2m} \eps_{\fs(J)}^k F^{(k)}(x, \theta_J)\sum_{\ell=0}^{k}
\alpha_{\ell,I}\frac{h_I^{k-\ell}}{(k-\ell)!}+ \eps_{\fs(J)}^{2m}\, \max_{\ell\le 2m}|\alpha_{\ell,I} | &\,o(1)\\
& + R(x,I).\\
\end{align*}
Adding up $F(x,I)$ over the children $I$ of $J$ gives \eqref{hyprec} i.e.
\begin{equation*}
  F(x,J)=\sum_{k=0}^{2m}\eta_{k,J} \eps_{\fs(J)}^k F^{(k)}(x, \theta_J)+R(x,J) ,
\end{equation*}
with 
\begin{equation}
\label{etakJ}
  \eta_{k,J}=\sum_{I\in \mathrm{Child}(J)}\sum_{\ell=0}^{k}
  \alpha_{\ell,I}\frac{h_I^{k-\ell}}{(k-\ell)!}
\end{equation}
and 
\begin{equation}
\label{RxJ}
  R(x,J)=\sum_{I\in \mathrm{Child}(J)}\left[\eps_{\fs(J)}^{2m}\, \max_{\ell\le 2m}|\alpha_{\ell,I} | \,o(1) + R(x,I)\right].
\end{equation}

\emph{We first prove \eqref{coeffsbornes} for $J$.} From
 \eqref{etakJ} for $k=0$ and \eqref{etalIhI} and recurrence hypothesis on $I$, we have 
\[\eta_{0,J}=\sum_{I\in \mathrm{Child}(J)}\alpha_{0,I}=\sum_{I\in \mathrm{Child}(J)}\eta_{0,I}=\sum_{I\in \mathrm{Child}(J)}\sum_{j\in I}\pi_j=\sum_{j\in J}\pi_j.\]
Moreover, since $|h_I|\lec 1$
for each child $I$ of $J$, Equation \eqref{etakJ} yields 
\[|\eta_{k,J}|\lec \max_{\substack{\ell\le k \\ I\in\mathrm{Child}(J)}} |\alpha_{\ell,I}|.\]
 Furthermore, from \eqref{etalIhI} we have $|\alpha_{\ell,I}| \le |\eta_{\ell,I}|$ since $\eps_{\fs(I)}\le
 \eps_{\fs(J)}$. By assumption on $I$, we have $|\eta_{\ell,I}|\lec 1$ so that
 $|\alpha_{\ell,I}|$ and thus $|\eta_{k,J}|$ are $O(1)$ and \eqref{coeffsbornes} is established. 

 \emph{We turn to the proof of \eqref{premierscoeffs}.} The first step is to show that
\begin{equation}
  \label{eq:applilm}
  \max_{k<d_J}|\eta_{k,J}|\asymp \max_{\substack{\ell< d_I \\ I\in \mathrm{Child}(J)}} |\alpha_{\ell,I}|.
\end{equation}
To this end, note that  \eqref{etalIhI} gives, for any two distinct
children $I$ and $I'$ of~$J$:
\begin{equation}
  \label{sep}
|h_{I} - h_{I'}| = \eps_{\fs(J)}^{-1} |\theta_{I} - \theta _{I'}|\asymp 1.
\end{equation}
The finite set $\{h_I\}_{I\in\C{J}}$ is thus $\eps$-separated for some $\eps>0$. Hence,  if we set $\Lambda=\left(\alpha_{\ell,I}\right)_{0\le \ell\le d_I-1}$, we get by Corollary~\ref{memenorme}
\[\max_{k<d_J}\left|\sum_{I\in \C{J}}\sum_{\ell=0}^{d_I-1}\alpha_{\ell,I}\frac{h_I^{k-\ell}}{(k-\ell)!}\1_{k\ge \ell}\right|\asymp \max_{\substack{\ell< d_I \\ I\in \C{J}}} |\alpha_{\ell,I}|.\]
Now, to obtain \eqref{eq:applilm}, we see from \eqref{etakJ} that it's enough to show 
\begin{equation}
  \label{eq:negl}
I\in \C{J}, k\ge d_I\Longrightarrow \sum_{\ell=d_I}^{k}\alpha_{\ell,I}\frac{h_I^{k-\ell}}{(k-\ell)!}=o\left(\max_{\ell< d_I} |\alpha_{\ell,I}|\right). 
\end{equation}
Since $|h_I|\lec 1$, we have 
\begin{equation}
  \label{eq:gro}
 \left|\sum_{\ell=d_I}^{k}\alpha_{\ell,I}\frac{h_I^{k-\ell}}{(k-\ell)!}\right|\lec \max_{d_I\le \ell\le k} |\alpha_{\ell,I}|.
\end{equation}
By assumption on $I$, we also have $\|\eta_I\|\asymp \max_{\ell< d_I} |\eta_{\ell,I}|$, so that
\begin{eqnarray*}
 \frac{\eps_{\fs(I)}}{\eps_{\fs(J)}}\cdot \max_{\ell< d_I} |\alpha_{\ell,I}| = \frac{\eps_{\fs(I)}}{\eps_{\fs(J)}}\cdot \max_{\ell< d_I} |\eta_{\ell,I}|\left(\frac{\eps_{\fs(I)}}{\eps_{\fs(J)}}\right)^\ell
& \gec & \|\eta_I\|\left(\frac{\eps_{\fs(I)}}{\eps_{\fs(J)}}\right)^{d_I} \\
&\ge & \max_{d_I\le \ell\le k} |\alpha_{\ell,I}| ,
\end{eqnarray*}
where the last inequality comes from \eqref{etalIhI}. Thus,  
\begin{equation}
  \label{eq:pto}
 \max_{d_I\le \ell\le k} |\alpha_{\ell,I}|= o\left(\max_{\ell< d_I} |\alpha_{\ell,I}|\right) ,
\end{equation}
so that \eqref{eq:gro} and \eqref{eq:pto} yield \eqref{eq:negl} and \eqref{eq:applilm} is proved. 

The second step is to prove 
\begin{equation}
  \label{equi}
 \|\eta_J\|\asymp \max_{k< d_J} |\eta_{k,J}|. 
\end{equation}
The non-trivial part is $\|\eta_J\|\lec \max_{k< d_J}
|\eta_{k,J}|$. By the definition \eqref{etakJ} of $\eta_{k,J}$,
\eqref{eq:pto} and \eqref{eq:applilm}, we have 
\begin{eqnarray*}
\max_{k\ge d_J} |\eta_{k,J}|  &\le  &  \max_{k\ge d_J}\sum_{I\in
                                      \C{J}}\max_{\ell\le k}
                                      |\alpha_{\ell,I}|\\
& \lec  &  \sum_{I\in \mathrm{Child}(J)}\max_{\ell<d_I} |\alpha_{\ell,I}| \lec   \max_{\substack{\ell< d_I \\ I\in \mathrm{Child}(J)}} |\alpha_{\ell,I}| \lec  \max_{k< d_J} |\eta_{k,J}|.
\end{eqnarray*}
The proof of \eqref{premierscoeffs} is complete.

\emph{We turn to the proof of \eqref{borneinfcoeff}.} From \eqref{equi},
\eqref{eq:applilm} and \eqref{etalIhI}, we get 
\begin{eqnarray}
  \|\eta_J\| \gec \max_{k< d_J} |\eta_{k,J}| \gec \max_{\substack{\ell< d_I \\ I\in \mathrm{Child}(J)}}|\alpha_{\ell,I}|
&\gec &\max_{\substack{\ell< d_I \\ I\in \mathrm{Child}(J)}}|\eta_{\ell,I}|\left(\frac{\eps_{\fs(I)}}{\eps_{\fs(J)}}\right)^\ell\nonumber\\
&\gec &\max_{I\in\mathrm{Child}(J)}\|\eta_I\|\left(\frac{\eps_{\fs(I)}}{\eps_{\fs(J)}}\right)^{d_I-1}.\label{mino}
\end{eqnarray}
Let $H$ be a descendant of $I$, a child of $J$. 
Assumption \eqref{borneinfcoeff} on $I$ then yields
\begin{eqnarray*}
  \|\eta_J\| \gec
  \|\eta_I\|\left(\frac{\eps_{\fs(I)}}{\eps_{\fs(J)}}\right)^{d_J-1} &\gec&
\eps_{\vs(H)}\left(\frac{\eps_{\fs(p(H))}}{\eps_{\fs(I)}}\right)^{d_I-1}\left(\frac{\eps_{\fs(I)}}{\eps_{\fs(J)}}\right)^{d_J-1}\\
&\gec& \eps_{\vs(H)}\left(\frac{\eps_{\fs(p(H))}}{\eps_{\fs(J)}}\right)^{d_J-1}.\\
\end{eqnarray*}
Moreover \eqref{coeffsbornes} implies 
\[\|\eta_J\|\gec |\eta_{0,J}|=|\pi(J)|\asymp \eps_{\vs(J)}.\]
Similarly, from \eqref{eq:applilm} and \eqref{coeffsbornes} for $I$,
\[\|\eta_J\|\gec |\alpha_{0,I}|=|\eta_{0,I}|=|\pi(I)|\asymp \eps_{\vs(I)},\]
so that \eqref{borneinfcoeff} is established for $J$.

\emph{We finally prove \eqref{remaind}.} From \eqref{RxJ}, \eqref{eq:pto}, assumption \eqref{remaind} for $I$ and \eqref{eq:applilm}, we have 
\begin{eqnarray*}
  R(x,J) &\lec & \max_{I\in\mathrm{Child}(J)}\left[\eps_{\fs(J)}^{2m}\, \max_{\ell\le
  2m}|\alpha_{\ell,I} | \,o(1) + R(x,I)\right]\\
 &\lec & \max_{I\in\mathrm{Child}(J)}\left[\eps_{\fs(J)}^{2m}\max_{\ell<d_I}|\alpha_{\ell,I} | \,o(1) + o\left(\|\eta_I\|\eps_{\fs(I)}^{2m}\right)\right]\\
 &\lec &\eps_{\fs(J)}^{2m}\|\eta_J\| \,o(1)+\max_{I\in\mathrm{Child}(J)}o\left(\|\eta_I\|\eps_{\fs(I)}^{2m}\right),
\end{eqnarray*}
and in addition, for each child $I$ of $J$, from \eqref{mino},
\begin{equation*}
  \|\eta_I\|\eps_{\fs(I)}^{2m}  = \|\eta_I\|\eps_{\fs(I)}^{d_I-1}\eps_{\fs(I)}^{2m+1-d_I}\lec  \|\eta_J\|\eps_{\fs(J)}^{d_I-1}\eps_{\fs(I)}^{2m+1-d_I}\le  \|\eta_J\|\eps_{\fs(J)}^{2m},
\end{equation*}
and we are done.

\bibliographystyle{imsart-nameyear}
\bibliography{minimax_mixtures_final}

%

\end{document}